\documentclass[a4paper,oneside]{amsart}
\usepackage{tikz}
\usepackage{bm} %bold symbols/letters in equations
\usepackage{listings}
\usepackage{enumitem} %environments for lists
\usepackage{wrapfig}
\usepackage{subcaption}
\usepackage{float}
\usepackage{amssymb} %For the varnothing symbol

\let\oldchi\chi
\renewcommand*{\chi}{\protect\raisebox{2pt}{$\oldchi$}}

\DeclareMathOperator{\Hom}{Hom} %For homomorphisms of maps

\newtheorem*{prop*}{Proposition}

\newtheorem*{lemma*}{Lemma}

\newtheorem*{theorem*}{Theorem}

\theoremstyle{remark}

\theoremstyle{definition}

\newtheorem*{defn*}{Definition}

\theoremstyle{definition}

\newtheorem*{exmp*}{Example}
\theoremstyle{definition}

\numberwithin{equation}{section}% standard style numbering, nothing special here

\newcommand{\doubleoverline}[1]{{\overline{\overline{#1}}}}

\title{Lowest-degree triple Massey products in moment-angle complexes}
\author{Jelena Grbi\'{c}}
\address{School of Mathematics, University of Southampton, UK}
\email{J.Grbic@soton.ac.uk}

\author{Abigail Linton}
\address{School of Mathematics, University of Southampton, UK}
\email{A.Linton@soton.ac.uk}

\begin{document}

\maketitle

%%%%%%%%%%%%%%%%%%%%%%%%%%%%%%%%%%%%%%%%%
%%%%%%%%%%%%%%%%%%%%%%%%%%%%%%%%%%%%%%%%%

\begin{abstract}
	We give a combinatorial classification of non-trivial triple Massey products of three dimensional classes in the cohomology of a moment-angle complex.
	This result improves on \cite[Theorem~6.1.1]{DenhamSuciu} by considering triple Massey products with non-trivial indeterminacy.
\end{abstract}

Massey products are higher cohomology operations that refine cup products. 
The first examples of non-trivial Massey products in moment-angle complexes were by Baskakov \cite{Baskakov}, who gave an infinite family of moment-angle complexes with non-trivial triple Massey products. 
For moment-angle complexes we give a classification of non-trivial triple Massey products of classes in the lowest degree (degree three).
Additionally this provides the smallest examples of Massey products in moment-angle complexes in which non-trivial indeterminacy can occur. 
This classification is important for recent results, see for example~\cite[Proposition 4.9]{rigidity}.

Let $\mathcal{K}$ be a simplicial complex on $[m]$ vertices. Following~\cite{ToricTopology}, the \textit{moment-angle complex} $\mathcal{Z}_\mathcal{K}$ is 
	\[
	 \mathcal{Z}_\mathcal{K}=\bigcup\limits_{I\in \mathcal{K}} \left( D^2, S^1 \right)^I \subset (D^2)^m
	\]
	where 
	$
	(D^2, S^1)^I = %\{ (x_1, ..., x_m)\in  (D^2)^m \colon x_i\in S^1 \text{ for } i\notin I\}.
	\prod_{i=1}^{m} Y_i
	$
	for $Y_i=D^2$ if $i\in I$, and $Y_i=S^1$ if~$i\notin I$.

As a subspace of the polydisc, $\mathcal{Z}_\mathcal{K}$ has a cellular decomposition that induces a multigrading on $C^*(\mathcal{Z}_\mathcal{K})$. 
%There is a cellular decomposition of $\mathcal{Z}_\mathcal{K}$ that gives a multigrading on $C^*(\mathcal{Z}_\mathcal{K})$. 
%Let $\mathbf{k}$ be a field or $\mathbb{Z}$. Coefficients are in $\mathbf{k}$, which is a field or $\mathbb{Z}$.
%\cite{BaskakovProduct, ToricTopology, Hochster}%]\label{thm: full Hochster's}
The \textit{full subcomplex} $\mathcal{K}_J$ is $\{\sigma\in \mathcal{K} \ | \ \sigma\subset J\}$ for $J\subset [m]$.
Let $\widetilde{C}^{*}(\mathcal{K}_J)$ be the augmented simplicial cochain complex. 
The following theorem is a combination of results by Hochster \cite{Hochster}, Buchstaber-Panov \cite{TorusActionsCombinatorics}, and Baskakov \cite{BaskakovProduct}.
All coefficients are in $\mathbf{k}$, which is a field or $\mathbb{Z}$. 
\begin{theorem*}[Hochster's formula]\cite{Hochster, TorusActionsCombinatorics, BaskakovProduct}
	There is an isomorphism of cochains $\widetilde{C}^{*-1}(\mathcal{K}_J)\to C^{*-|J|, 2J}(\mathcal{Z}_\mathcal{K})\subset C^{*+|J|}(\mathcal{Z}_\mathcal{K})$ that induces an isomorphism of algebras 
	$
	H^*(\mathcal{Z}_\mathcal{K}) \cong \bigoplus_{J \subset [m]} \widetilde{H}^*(\mathcal{K}_J),
	$
	where $\widetilde{H}^{-1}(\mathcal{K}_{\varnothing})=\mathbf{k}$.
\end{theorem*}
	
The cochain group $C^{p}(\mathcal{K}_J)=\Hom (C_p(\mathcal{K}_J), \mathbf{k})$ has a basis of $\chi_L$ for a $p$-simplex $L \in \mathcal{K}_J$,  where $\chi_L$ takes the value $1$ on $L$ and $0$ otherwise. %on any other simplex of $\mathcal{K}_J$. 
Let $\varepsilon (j, J)=(-1)^{r-1}$ for $j$ the $r$th element of $J$, and for $L\subset J$, let
$
\varepsilon(L, J)=\prod_{j\in L} \varepsilon (j, J)$.
The product on $\bigoplus_{J\subset [m]} \widetilde{H}^*(\mathcal{K}_J)$ is induced by $C^{p-1} (\mathcal{K}_I) \otimes C^{q-1} (\mathcal{K}_J) \to C^{p+q-1} (\mathcal{K}_{I\cup J})$,
%For two simplices $L=\{l_0, \ldots, l_p \}, M=\{m_0, \ldots, m_q \}\in \mathcal{K}$, $L\cap M=\myempty$, let $L\cup M$ denote the simplex on the vertices $\{l_0, \ldots, l_p, m_0, \ldots, m_q\}$. Then there is a product on $\bigoplus_{J\subset [m]} \widetilde{H}^*(\mathcal{K}_J)$ given by Baskakov \cite{BaskakovProduct},
\begin{equation*} \label{eq: cochain product}
\chi_L \otimes \chi_M \mapsto 
\left\{ \begin{array}{lr}
c_{L\cup M}\; \chi_{L\cup M} & \text{if } I\cap J = \varnothing,  \\
0 & \text{otherwise} 
\end{array} \right.
\end{equation*}
where
$
c_{L\cup M} = \varepsilon(L,I) \;\varepsilon(M,J) \;\zeta\; \varepsilon(L\cup M, I\cup J)
$
and $
\zeta= \prod_{k\in I\setminus L} \varepsilon(k,k\cup J\setminus M)$.

\begin{defn*}
	Let $(A, d)$ be a differential graded algebra with $\alpha_i\in H^{p_i}(A)$ for $i=1,2,3$ such that $\alpha_1 \alpha_2=0$ and $\alpha_2 \alpha_3=0$. For $a\in A^p$, let $\overline{a}=(-1)^{1+p}a$ and let $a_i\in A^{p_i}$ be a representative for $\alpha_i$. 
	The {\it triple Massey product} $\langle \alpha_1, \alpha_2, \alpha_3 \rangle \subset H^{p_1+p_2+p_3-1}[A]$ is the set of classes represented by 
	$\overline{a}_{1}a_{23}+ \overline{a}_{12}a_{3} \in A^{p_1+p_2+p_3-1}$,
	for $a_{i,i+1} \in A^{p_{i}+p_{i+1}-1}$ such that $d(a_{i,i+1})=\overline{a}_{i}a_{i+1}$, $i=1, 2$ and it does not depend on the representative $a_i$ for $\alpha_i$. 
	A triple Massey product is \textit{trivial} if it contains $0$.
The \textit{indeterminacy} of a triple Massey product is the set of differences between elements in $\langle \alpha_1, \alpha_2, \alpha_3 \rangle$.
\end{defn*}

We compute Massey products $\langle \alpha_1, \alpha_2, \alpha_3 \rangle \subset H^*(\mathcal{Z}_\mathcal{K})$ in terms of $\mathcal{K}$ by using the isomorphism in Hochster's formula.
When indeterminacy is trivial, Denham and Suciu \cite{DenhamSuciu} showed that there is a non-trivial Massey product $\langle \alpha_1,\alpha_2, \alpha_3 \rangle$ for $\alpha_i\in H^3(\mathcal{Z}_\mathcal{K})$ if and only if the one-skeleton $\mathcal{K}^{(1)}$ contains a full subcomplex isomorphic to one of the first $6$ graphs in Figure~\ref{fig: obstruction graphs}. 
Those six graphs do not capture Massey products with non-trivial indeterminacy.

\begin{exmp*} \label{ex: Non-trivial indeterminacy triple Massey example}
	Let $\mathcal{K}$ be the graph in Figure~\ref{fig: both examples}b, where the dashed edge $\{4,6\}$ is optional. 
	Let $\alpha_1, \alpha_2, \alpha_3 \in H^3(\mathcal{Z}_\mathcal{K})$ correspond to $\alpha_1=[\chi_1]\in \widetilde{H}^0(\mathcal{K}_{12})$, $\alpha_2=[\chi_3]\in \widetilde{H}^0(\mathcal{K}_{34})$, $\alpha_3=[\chi_5] \in \widetilde{H}^0(\mathcal{K}_{56})$.
	Since $\widetilde{H}^1(\mathcal{K}_{1234})=0$ and $\widetilde{H}^1(\mathcal{K}_{3456})=0$,  the products $\alpha_1\alpha_2\in\widetilde{H}^1(\mathcal{K}_{1234})$ and $\alpha_2\alpha_3\in$ $\widetilde{H}^1(\mathcal{K}_{3456})$ are zero.
	
	For $a\in C^p(\mathcal{K}_J)$, which corresponds to $a\in C^{p+|J|+1}(\mathcal{Z}_\mathcal{K})$, denote $\doubleoverline{a}=(-1)^{p+|J|}a$.
	A cochain $a_{12}\in C^0(\mathcal{K}_{1234})$ such that $d(a_{12})=\doubleoverline{\chi_1} \chi_3=0$ is of the form $a_{12}=c_1 \chi_3+c_2(\chi_1+\chi_4+\chi_2)$ for any $c_1, c_2\in \mathbf{k}$.
	A cochain $a_{23}\in C^0(\mathcal{K}_{3456})$ such that $d(a_{23})=\doubleoverline{\chi_3}\cdot \chi_5=\chi_{35}$ is of the form $a_{23}=c_3\chi_4 +c_4(\chi_6+\chi_3+\chi_5) +\chi_5$ for any $c_3, c_4\in \mathbf{k}$, where $c_3=c_4$ if $\{4,6\}\in \mathcal{K}$.
%	We will find all cochains $a_{12}\in C^0(\mathcal{K}_{1234})$ such that $d(a_{12})=\doubleoverline{\chi_1} \chi_3=0$. 
%	For a general cochain $x=c_1 \chi_3 + c_2 \chi_1 +c_3\chi_4+c_4\chi_2\in C^0(\mathcal{K}_{1234})$ with coefficients $c_1, \ldots, c_4$, $d(x)=0-c_2\chi_{14}+c_3 \chi_{14}+c_3 \chi_{24} -c_4\chi_{24}$. 
%	So $d(x)=0$ if $c_2=c_3=c_4$.
%	Hence $a_{12}=c_1 \chi_3+c_2(\chi_1+\chi_4+\chi_2)\in C^0(\mathcal{K}_{1234})$ for any coefficients $c_1, c_2$. 
%	We also find all cochains $a_{23}\in C^0(\mathcal{K}_{3456})$ such that $d(a_{23})=\doubleoverline{\chi_3}\cdot \chi_5=\chi_{35}$. 
%	For a general cochain in $y=c_1 \chi_4 +c_2\chi_6 + c_3 \chi_3 + c_4 \chi_5\in C^0(\mathcal{K}_{3456})$ with coefficients $c_1, \ldots, c_4$, 
%	$d(y)=-c_1 \chi_{46}+c_2\chi_{46}+c_2\chi_{36}-c_3\chi_{36} -c_3\chi_{35}+c_4\chi_{35}$.
%	If $\{4,6\}\notin \mathcal{K}$, then $\chi_{46}=0$ so $d(y)=\chi_{35}$ when $c_2=c_3=c_4-1$.
%	If $\{4,6\}\in \mathcal{K}$, then $d(y)=\chi_{35}$ when $c_1=c_2=c_3=c_4-1$.
%	Therefore $a_{23}=c_3\chi_4 +c_4(\chi_6+\chi_3+\chi_5) +\chi_5\in  C^0(\mathcal{K}_{3456})$ for coefficients $c_3, c_4$, where $c_3=c_4$ if $\{4,6\}\in \mathcal{K}$. \todo{Not sure about coefficient labelling in $y$}
	The associated cocycle $\omega\in C^1(\mathcal{K})$ is 
	$\omega=\doubleoverline{a}_1a_{23}+\doubleoverline{a}_{12}a_3=
%	\chi_1(c_3\chi_4 +c_4(\chi_6+\chi_3+\chi_5) +\chi_5)+(c_1 \chi_3+c_2(\chi_1+\chi_4+\chi_2))\chi_5=
	c_3 \chi_{14} + c_4(\chi_{16}+\chi_{15})+\chi_{15}+c_1 \chi_{35} + c_2(\chi_{15}+\chi_{25})
	$.
	Since $d(\chi_5)=\chi_{15}+\chi_{35}+\chi_{25}$ and $d(\chi_1)=-\chi_{16}-\chi_{14}-\chi_{15}$ for $\chi_1, \chi_5\in C^0(\mathcal{K})$, $\omega=(c_3-c_4)\chi_{14}-c_4d(\chi_1)+\chi_{15}+(c_1-c_2)\chi_{35}+c_2d(\chi_5)$.
	Also $[\omega]=[(c_3-c_4)\chi_{14}+\chi_{15}+(c_1-c_2)\chi_{35}]\neq 0$ for any $c_1, c_2, c_3, c_4\in \mathbf{k}, c_3=c_4 \text{ if } \{4,6\}\in \mathcal{K}$. 
	Therefore $\langle \alpha_1, \alpha_2, \alpha_3 \rangle\subset H^8(\mathcal{Z}_\mathcal{K})$ is non-trivial with non-trivial indeterminacy, 
	$ 
	\alpha_1 \cdot \widetilde{H}^0(\mathcal{K}_{3456}) + \alpha_3 \cdot \widetilde{H}^0(\mathcal{K}_{1234}). 
	$

	Calculations are similar for $\mathcal{K}$ in Figure~\ref{fig: both examples}a, where
	$a_{12}=\chi_3+c_1(\chi_1+\chi_2+\chi_3+\chi_4)$ and $a_{23}=\chi_5+c_2(\chi_3+\chi_4+\chi_5+\chi_6)$, $c_1, c_2\in \mathbf{k}$.
	In this case, $\omega=-\chi_{25}+(c_1+1)d(\chi_5)-c_2d(\chi_1)$ so $\langle \alpha_1, \alpha_2, \alpha_3 \rangle$ only contains the non-zero class $[\omega]=[-\chi_{25}]$.
	% in this case, $\alpha_1 \cdot \widetilde{H}^0(\mathcal{K}_{3456}) + \alpha_3  \cdot \widetilde{H}^0(\mathcal{K}_{1234})=0$ so there is only one element in $\langle \alpha_1, \alpha_2, \alpha_3 \rangle \subset H^8(\mathcal{Z}_\mathcal{K})$.
\end{exmp*}

\begin{theorem*} \label{thm: obstruction graphs}
	There is a non-trivial triple Massey product $\langle \alpha_1, \alpha_2, \alpha_3 \rangle\subset H^8(\mathcal{Z}_\mathcal{K})$ for $\alpha_1, \alpha_2, \alpha_3 \in H^3(\mathcal{Z}_\mathcal{K})$ if and only if the one-skeleton $\mathcal{K}^{(1)}$ of $\mathcal{K}$ contains a full subcomplex isomorphic to a graph in Figure~\ref{fig: obstruction graphs}. 
\end{theorem*}

% My version of Denham & Suciu's graphs
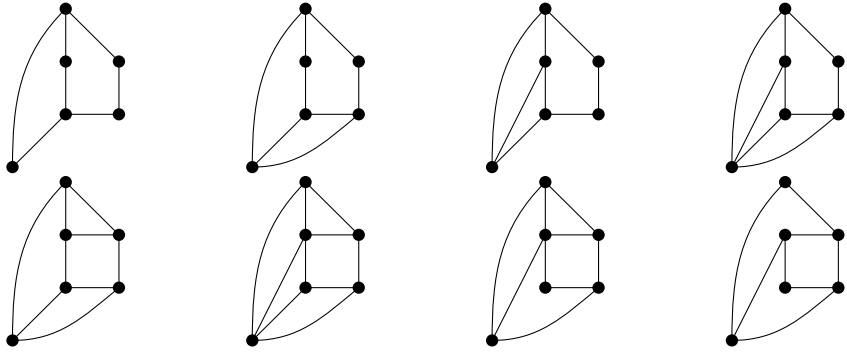
\begin{figure}[h]
	\centering
	\begin{minipage}{0.24\textwidth}
		\centering
		\begin{tikzpicture}
		\coordinate (a) at (0,0);
		\coordinate (b) at (0.7,0);
		\coordinate (c) at (0,0.7);
		\coordinate (d) at (0.7,0.7);
		\coordinate (e) at (0,1.4);
		\coordinate (f) at (-0.7,-0.7);
		
		\draw (f) -- (a) -- (b) -- (d) --(e) -- (c) -- (a);
		\draw (f)  to [out=90,in=-135] (e);
		
		\fill (a) circle (2.3pt);\fill (b) circle (2.3pt);\fill (c) circle (2.3pt);
		\fill (d) circle (2.3pt);\fill (e) circle (2.3pt);\fill (f) circle (2.3pt);
		\end{tikzpicture} 
	\end{minipage}
	\begin{minipage}{0.24\textwidth}
		\centering
		\begin{tikzpicture}
		\coordinate (a) at (0,0);
		\coordinate (b) at (0.7,0);
		\coordinate (c) at (0,0.7);
		\coordinate (d) at (0.7,0.7);
		\coordinate (e) at (0,1.4);
		\coordinate (f) at (-0.7,-0.7);
		
		\draw (f) -- (a) -- (b) -- (d) --(e) -- (c) -- (a);
		\draw (f) to [out=90,in=-135] (e);
		\draw (f) to [out=0,in=-140] (b);
		
		\fill (a) circle (2.3pt);\fill (b) circle (2.3pt);\fill (c) circle (2.3pt);
		\fill (d) circle (2.3pt);\fill (e) circle (2.3pt);\fill (f) circle (2.3pt);
		\end{tikzpicture} 
	\end{minipage}
	\begin{minipage}{0.24\textwidth}
		\centering
		\begin{tikzpicture}
		\coordinate (a) at (0,0);
		\coordinate (b) at (0.7,0);
		\coordinate (c) at (0,0.7);
		\coordinate (d) at (0.7,0.7);
		\coordinate (e) at (0,1.4);
		\coordinate (f) at (-0.7,-0.7);
		
		\draw (f) -- (a) -- (b) -- (d) --(e) -- (c) -- (a);
		\draw (f) to [out=90,in=-135] (e);
		\draw (f) -- (c);
		
		\fill (a) circle (2.3pt);\fill (b) circle (2.3pt);\fill (c) circle (2.3pt);
		\fill (d) circle (2.3pt);\fill (e) circle (2.3pt);\fill (f) circle (2.3pt);
		\end{tikzpicture} 
		
	\end{minipage}
	\begin{minipage}{0.24\textwidth}
		\centering
		\begin{tikzpicture}
		\coordinate (a) at (0,0);
		\coordinate (b) at (0.7,0);
		\coordinate (c) at (0,0.7);
		\coordinate (d) at (0.7,0.7);
		\coordinate (e) at (0,1.4);
		\coordinate (f) at (-0.7,-0.7);
		
		\draw (f)-- (a) -- (b) -- (d) --(e) -- (c) -- (a);
		\draw (f)  to [out=90,in=-135] (e);
		\draw (f) to [out=0,in=-140] (b);
		\draw (f) -- (c);
		
		\fill (a) circle (2.3pt);\fill (b) circle (2.3pt);\fill (c) circle (2.3pt);
		\fill (d) circle (2.3pt);\fill (e) circle (2.3pt);\fill (f) circle (2.3pt);
		\end{tikzpicture} 
	\end{minipage}\\
	\begin{minipage}{0.24\textwidth}
		\centering
		\begin{tikzpicture}
		\coordinate (a) at (0,0);
		\coordinate (b) at (0.7,0);
		\coordinate (c) at (0,0.7);
		\coordinate (d) at (0.7,0.7);
		\coordinate (e) at (0,1.4);
		\coordinate (f) at (-0.7,-0.7);
		
		\draw (f) -- (a) -- (b) -- (d) --(e) -- (c) -- (a);
		\draw (f)  to [out=90,in=-135] (e);
		\draw (f) to [out=0,in=-140] (b);
		\draw (c) -- (d);
		
		\fill (a) circle (2.3pt);\fill (b) circle (2.3pt);\fill (c) circle (2.3pt);
		\fill (d) circle (2.3pt);\fill (e) circle (2.3pt);\fill (f) circle (2.3pt);
		\end{tikzpicture} 
	\end{minipage}
	\begin{minipage}{0.24\textwidth}
		\centering
		\begin{tikzpicture}
		\coordinate (a) at (0,0);
		\coordinate (b) at (0.7,0);
		\coordinate (c) at (0,0.7);
		\coordinate (d) at (0.7,0.7);
		\coordinate (e) at (0,1.4);
		\coordinate (f) at (-0.7,-0.7);
		
		\draw (f) -- (a) -- (b) -- (d) --(e) -- (c) -- (a);
		\draw (f)  to [out=90,in=-135] (e);
		\draw (f) to [out=0,in=-140] (b);
		\draw (f) -- (c) -- (d);
		
		\fill (a) circle (2.3pt);\fill (b) circle (2.3pt);\fill (c) circle (2.3pt);
		\fill (d) circle (2.3pt);\fill (e) circle (2.3pt);\fill (f) circle (2.3pt);
		\end{tikzpicture} 
	\end{minipage}
	\begin{minipage}{0.24\textwidth}
		\centering
		\begin{tikzpicture}
		\coordinate (a) at (0,0);
		\coordinate (b) at (0.7,0);
		\coordinate (c) at (0,0.7);
		\coordinate (d) at (0.7,0.7);
		\coordinate (e) at (0,1.4);
		\coordinate (f) at (-0.7,-0.7);
		
		\draw (a) -- (b) -- (d) --(e) -- (c) -- (a);
		\draw (f)  to [out=90,in=-135] (e);
		\draw (f) to [out=0,in=-140] (b);
		\draw (f) -- (c) -- (d);
		
		\foreach \i in {a, b, c, d, e, f} {\fill (\i) circle (2.3pt);}
		\end{tikzpicture} 
	\end{minipage}
	\begin{minipage}{0.24\textwidth}
		\centering
		\begin{tikzpicture}
		\coordinate (a) at (0,0);
		\coordinate (b) at (0.7,0);
		\coordinate (c) at (0,0.7);
		\coordinate (d) at (0.7,0.7);
		\coordinate (e) at (0,1.4);
		\coordinate (f) at (-0.7,-0.7);
		
		\draw (a) -- (b) -- (d) --(e); \draw (c) -- (a);
		\draw (f)  to [out=90,in=-135] (e);
		\draw (f) to [out=0,in=-140] (b);
		\draw (f) -- (c) -- (d);
		
		\foreach \i in {a, b, c, d, e, f} {\fill (\i) circle (2.3pt);}
		\end{tikzpicture} 
	\end{minipage}
	\caption{The eight obstruction graphs}
	\label{fig: obstruction graphs}
\end{figure}

\begin{figure}[h!]
	\centering
	\begin{minipage}{0.4\textwidth}
		\centering
		\begin{tikzpicture}[scale=1, inner sep=2mm]
		\coordinate (3) at (0,0);
		\coordinate (5) at (1,0);
		\coordinate (6) at (0,1);
		\coordinate (2) at (1,1);
		\coordinate (4) at (0,2);
		\coordinate (1) at (-1,-1);
		\coordinate (a) at (-1.5, 0.5);
		
		% Draw the edges
		\draw (1) -- (3) -- (5) -- (2) --(4) -- (6)  -- (3);
		\draw (1)  to [out=90,in=-135] (4);
		\draw[dashed] (1) to [out=0,in=-140] (5);
		\draw[dashed] (1) -- (6);
		\draw[dashed] (2)--(6);
		
		% Draw all of the vertices
		\foreach \i in {1, ..., 6} {\fill (\i) circle (2.5pt);}
		
		% Label the vertices
		\draw (1) node[left] {$1$};
		\draw (2) node[right] {$2$};
		\draw (3) node[below] {$3$};
		\draw (4) node[right] {$4$};
		\draw (5) node[right] {$5$};
		\draw (6) node[left] {$6$};
		\draw (a) node[left] {(a)};
		\end{tikzpicture} 
		\label{fig: trivial indeterminacy example}
	\end{minipage} 
	\begin{minipage}{0.4\textwidth}
		\centering
		\begin{tikzpicture}[scale=1, inner sep=2mm]
		\coordinate (3) at (0,0);
		\coordinate (5) at (1,0);
		\coordinate (6) at (0,1);
		\coordinate (2) at (1,1);
		\coordinate (4) at (0,2);
		\coordinate (1) at (-1,-1);
		\coordinate (a) at (-1.5, 0.5);
		
		% Draw the edges
		\draw (2)  -- (6)--(3) -- (5) -- (2) --(4) ;
		\draw (1)  to [out=90,in=-135] (4);
		\draw (1) to [out=0,in=-140] (5);
		\draw (1) -- (6);
		\draw[dashed] (4)--(6);
		
		% Draw all of the vertices
		\foreach \i in {1, ..., 6} {\fill (\i) circle (2.5pt);}
		
		% Label the vertices
		\draw (1) node[left] {$1$};
		\draw (2) node[right] {$2$};
		\draw (3) node[below] {$3$};
		\draw (4) node[right] {$4$};
		\draw (5) node[right] {$5$};
		\draw (6) node[left] {$6$};
		\draw (a) node[left] {(b)};
		\end{tikzpicture} 
		\label{fig: non-trivial indeterminacy example}
	\end{minipage}\qquad	
	\caption{Massey products with trivial (a) and non-trivial (b) indeterminacy.}
	\label{fig: both examples}
\end{figure}
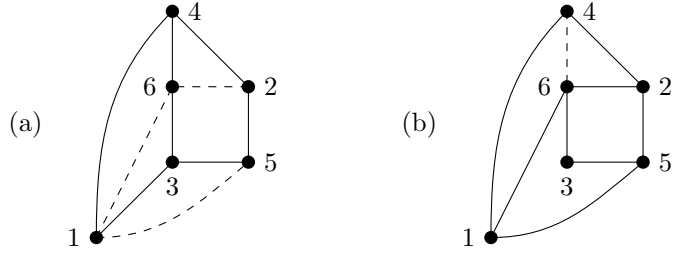

\begin{proof}
	As $\mathcal{Z}_{\mathcal{K}_J}$ retracts off $\mathcal{Z}_\mathcal{K}$ \cite{stephensnotes}, it is sufficient to prove the Theorem when $\mathcal{K}$ has six vertices. 
	Let $\mathcal{K}^{(1)}$ be a graph in Figure~\ref{fig: obstruction graphs}.  Then $\langle \alpha_1, \alpha_2, \alpha_3 \rangle$ is~{non-trivial} as
	calculations in the Example are not affected by $2$-simplices in $\mathcal{K}$, and $\dim(\mathcal{K})\leqslant 2$.

	Conversely, suppose $\langle \alpha_1, \alpha_2, \alpha_3 \rangle$ is non-trivial for $\alpha_1, \alpha_2, \alpha_3 \in H^{3}(\mathcal{Z}_\mathcal{K})$. 
	Let $G$ on $[6]$ be the edge complement graph of $\mathcal{K}^{(1)}$, so $\{i,j\}\in G$ if and only if $\{i,j\}\notin \mathcal{K}$.
	We will show that $G$ is one of the graphs in Figure~\ref{fig: both examples}, where the dashed edges are optional. 
	By Hochster's formula, there are full subcomplexes $\mathcal{K}_{S_i}$ for $S_i\subset [m]$, $|S_i|=2$ such that $\alpha_i$ corresponds to $\alpha_i\in \widetilde{H}^0(\mathcal{K}_{S_i})$. 
	Since $\mathcal{K}_{S_i}$ is a pair of disjoint vertices, $\{v_i, v_i'\}\notin \mathcal{K}$ for any $v_i, v_i'\in S_i$.
	Since $\langle \alpha_1, \alpha_2, \alpha_3 \rangle$ is non-trivial, $S_i\cap S_j=\varnothing$ for $i\neq j$. 
	Let $S_1=\{1,2\}, S_2=\{3,4\}, S_3=\{5,6\}$. 
	Then the graph $G$ contains the edges $\{1, 2\}, \{3, 4\}, \{5, 6\}$. 
	Since $\alpha_i \alpha_{i+1}=0$, the full subcomplex $\mathcal{K}_{S_i\cup S_{i+1}}$ does not contain a cycle  for $i=1,2$.
	Thus there exist edges 
	$\{v_1, v_2\}, \{v_2', v_3\}\in G$ for $v_i, v_i'\in S_i$. 
	
	Suppose $\{v_1,v_2\}, \{v_2, v_3\}\in G$ for $v_i\in S_i$. 
	Let $a_1=\chi_{v_1}\in C^0(\mathcal{K}_{12})$, $a_2=\chi_{v_2}\in C^0(\mathcal{K}_{34})$, $a_3=\chi_{v_3}\in C^0(\mathcal{K}_{56})$ be representing cocycles for $\alpha_1, \alpha_2, \alpha_3$. 
	So $a_1a_2=\chi_{v_1 v_2}=0$ and $a_2a_3=\chi_{v_2v_3}=0$. 
	%So $a_1a_2=\chi_{v_1}\chi_{v_2}=0$ and $a_2a_3=\chi_{v_2}\chi_{v_3}=0$. 
	For $a_{12}=0=a_{23}$, $\omega=0$, % the associated cocycle is zero, 
	which contradicts the non-triviality of $\langle \alpha_1, \alpha_2, \alpha_3 \rangle$. 
	Thus $\{v_1,v_2\}, \{v_2, v_3\}\notin G$ for $v_i\in S_i$. 
	
	\begin{figure}[h]
		\centering
		\begin{minipage}{0.4\textwidth}
			\centering
			\begin{tikzpicture}[scale=1, inner sep=2mm]
			\coordinate (a) at (0,0);
			\coordinate (b) at (0,1);
			\coordinate (c) at (0,2);
			\coordinate (d) at (1,2);
			\coordinate (e) at (1,1);
			\coordinate (f) at (1,0);
			\coordinate (label) at (-0.8, 1);
			
			\draw (a) node[left] {1} -- (b) node[left]{2} -- (c)node[left]{3} -- (d) node[right]{4} --(e) node[right]{5} -- (f) node[right]{6};
			\draw[dashed] (b) -- (f) -- (a) -- (e);
			\draw (label) node[left] {(a)};
			
			% Draw all of the vertices
			\foreach \i in {a, ..., f} {\fill (\i) circle (2.5pt);}
			\end{tikzpicture} 
			%			\label{fig: edge complement for trivial indeterminacy}
		\end{minipage}
		\begin{minipage}{0.4\textwidth}
			\centering
			\begin{tikzpicture}[scale=1, inner sep=2mm]
			\coordinate (a) at (0,0);
			\coordinate (b) at (0,1);
			\coordinate (c) at (0,2);
			\coordinate (d) at (1,2);
			\coordinate (e) at (1,1);
			\coordinate (f) at (1,0);
			\coordinate (label) at (-1, 1);
			
			\draw (a) node[below left, inner sep=1mm] {1} -- (b) node[left]{2} -- (c)node[above left, inner sep=1mm]{3} -- (d) node[above right, inner sep=1mm]{4} --(e) node[right]{5} -- (f) node[below right, inner sep=1mm]{6};
			%			\draw (a) to [bend left=50] (c);
			\draw plot [smooth, tension=1.5] coordinates { (a) (-0.7,1) (c) };
			\draw[dashed] plot [smooth, tension=1.5] coordinates { (d) (1.7,1) (f) };
			
			\draw (label) node[left] {(b)};
			
			% Draw all of the vertices
			\foreach \i in {a, ..., f} {\fill (\i) circle (2.5pt);}
			\end{tikzpicture} 
			%			\label{fig: edge complement for non-trivial indeterminacy}
		\end{minipage}
		\caption{Edge complement graphs $G$,  dashed edges optional.}
		\label{fig: graph complement}
	\end{figure}
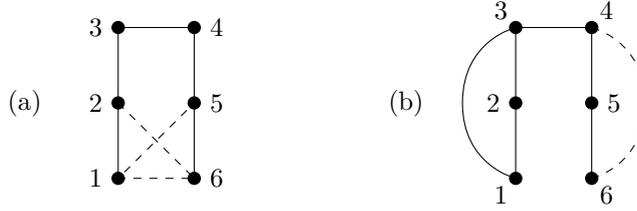

	Label the vertices of $G$ so that there is a path $1, \ldots,  6$. 
	Consider the case when $\{1,3\}, \{4,6\}\notin G$. Since $\{v_1,v_2\}, \{v_2, v_3\} \notin G$ for $v_i\in S_i$, the vertices $3$ and $4$ have valency two. 
	Suppose $\{2,5\}\in G$. 
	%Let $\alpha_1, \alpha_2, \alpha_3$ be represented by $\chi_2\in C^0(\mathcal{K}_{12}), \chi_3\in C^0(\mathcal{K}_{34}), \chi_5\in C^0(\mathcal{K}_{56})$, respectively. 
	Let $\chi_2\in C^0(\mathcal{K}_{12}), \chi_3\in C^0(\mathcal{K}_{34}), \chi_5\in C^0(\mathcal{K}_{56})$ represent $\alpha_1, \alpha_2, \alpha_3$, respectively. 
	Since $a_1a_2=0$, let $a_{12}=0$ and let $a_{23}=\chi_5$. 
	Then %the associated cocycle 
	$\omega=\chi_{25}$ is zero, contradicting the non-triviality of $\langle \alpha_1, \alpha_2, \alpha_3 \rangle$. 
	Hence $\{2,5\}\notin G$ and $G$ is the graph in Figure~\ref{fig: graph complement}a, 
	%$\mathcal{K}^{(1)}$ in Figure~\ref{fig: both examples}a
	where $\{1,5\}$, $\{1,6\}$ and $\{2,6\}$ are optional. 
	
	There are three more cases. 
	When $\{1,3\}\in G$ and $\{4,6\}\notin G$, it is necessary that $\{1,5\}\notin G$ otherwise in the Example, %the associated cocycle 
	$\omega=\chi_{15}$ is zero. 
	Also $\{2,5\}\notin G$ as in the previous case. 
	The edges $\{1,6\}, \{2,6\}$ in $G$ are optional as they do not change the calculations in the Example. 
	If $\{1,6\}, \{2,6\}\notin G$, then $G$ is the graph in Figure~\ref{fig: graph complement}b with $\{4,6\}\notin G$. 
	For other selections of $\{1,6\}, \{2,6\}$, $G$ is isomorphic to a graph in Figure~\ref{fig: graph complement}a. %$\mathcal{K}^{(1)}$ is graph isomorphic to graph (a). 
	%	When $\{1,6\}, \{2,6\}\in G$, $\mathcal{K}$ is isomorphic to the first graph in Figure~\ref{fig: obstruction graphs}. 
	%	If either $\{1,6\}\notin G$ and $\{2,6\}\in G$ or $\{2,6\}\notin G$ and $\{1,6\}\in G$, then $\mathcal{K}$ is isomorphic to the second graph in Figure~\ref{fig: obstruction graphs}. 
	%	When $\{1,6\}, \{2,6\}\notin G$, $\mathcal{K}$ is the seventh graph in Figure~\ref{fig: obstruction graphs} (the same as Figure~\ref{fig: discussionexample}). 
	
	When $\{4,6\}\in G$ and $\{1,3\}\notin G$, $G$ is symmetric to the case when $\{1,3\}\in G$ and $\{4,6\}\notin G$, so up to isomorphism we obtain the same graphs. 
	%	In the next case, where $\{4,6\}\in G$ and $\{1,3\}\notin G$, $G$ is symmetric to the case when $\{1,3\}\in G$ and $\{4,6\}\notin G$. That is, if $\{2,6\}\notin G$ and $\alpha_1, \alpha_2, \alpha_3$ are represented by $\chi_2\in C^0(\mathcal{K}_{12})$, $\chi_3\in C^0(\mathcal{K}_{24})$, $\chi_6\in C^0(\mathcal{K}_{56})$, respectively, then we can choose $a_{1,2}=0$, $a_{2,3}=\chi_6$ and the associated cocycle is $\omega=\chi_{26}=0$. 
	%	Thus in this case, we require the edge $\{2,6\}\in G$ and there are optional edges $\{1,5\}, \{1,6\}$. 
	%	Up to isomorphism, we obtain the same graphs as in the previous case. 
	Finally when $\{1,3\}, \{4,6\} \in G$, then $\{1,5\}, \{2,5\}, \{2,6\} \notin G$ for the same reasons as in the last two cases. 
	%	Suppose $\{1,6\}\in G$. 
	Let $\alpha_1, \alpha_2, \alpha_3$ be represented by $\chi_1\in C^0(\mathcal{K}_{12})$, $\chi_3\in C^0(\mathcal{K}_{34})$, $\chi_6\in C^0(\mathcal{K}_{56})$, respectively. 
	Let $a_{12}=0$, $a_{23}=\chi_6$. Then %associated cocycle is 
	$\omega=\chi_{16}$ and therefore $\{1,6\}\notin G$. 
	So $G$ is the graph in Figure~\ref{fig: graph complement}b with $\{4,6\}\in G$. 
	
	The graphs in Figure~\ref{fig: graph complement}a and Figure~\ref{fig: graph complement}b are edge complement graphs to the graphs in Figure~\ref{fig: both examples}a and Figure~\ref{fig: both examples}b, respectively.
	Up to graph isomorphism, all graphs in Figure~\ref{fig: both examples} are exactly those in Figure~\ref{fig: obstruction graphs}. 
	Thus Figure~\ref{fig: obstruction graphs} contains all $1$-skeletons of simplicial complexes $\mathcal{K}$ on six vertices with a non-trivial triple Massey product in $H^8(\mathcal{Z}_\mathcal{K})$. 
\end{proof}

\begin{lemma*}
	Not one of the graphs in Figure~\ref{fig: obstruction graphs} is isomorphic to another.
\end{lemma*}
\begin{proof}
	Label the eight graphs $a, b, c, d$ along the top row and $e, f, g, h$ along the bottom.
	Two graphs are not isomorphic if their vertices have different valencies. 
	For each graph, we list the valency of the vertices.
	\begin{align*}
		a: 3, 3, 2, 2, 2, 2 \qquad
		b: 3, 3, 3, 3, 2, 2 \qquad
		c: 3, 3, 3, 3, 2, 2 \qquad
		d: 4, 3, 3, 3, 3, 2 \\
		e: 3, 3, 3, 3, 3, 3 \qquad
		f: 4, 4, 3, 3, 3, 3 \qquad
		g: 4, 3, 3, 3, 3, 2 \qquad
		h: 3, 3, 3, 3, 2, 2
	\end{align*}
	Thus graphs $a, e, f$ are not isomorphic to any of the other graphs.
	Also the graphs $d$ and $g$ are not isomorphic because the vertices of valency $2$ and $4$ are adjacent in $g$ but not adjacent in $d$.
	The graph $c$ is different to $b, h$ because the two vertices of valency $2$ are adjacent in graph $c$ but not adjacent in $b$ or $h$.
	The graph $b$ is different to $h$ because the two vertices of valency $2$ are at a minimal distance of $2$ from each other, that is, there is one vertex in between them in $b$. In $h$, these two vertices are at a minimal distance of $3$ from each other.
	Therefore each of these graphs is not isomorphic to another.
\end{proof}

\bibliographystyle{abbrv}
\bibliography{minibibliography} %Can add other bibliographies by comma separating (no spaces)

\end{document}